\documentclass[a4paper,11pt]{amsart}

\usepackage[dvipdfmx]{graphicx}
\usepackage[dvipdfmx]{color}
\usepackage{amsmath,amssymb}
\usepackage{amsthm}
\usepackage{mathrsfs}
\usepackage{enumitem}
\usepackage{mathtools}
\usepackage[all,cmtip,2cell]{xy}
\usepackage{array}
\usepackage{here}
\usepackage{color}
\usepackage{fancybox}
\usepackage{chngcntr}
\counterwithout{equation}{section}

\theoremstyle{plain}
\newtheorem{thm}{Theorem}[section]
\newtheorem{prop}[thm]{Proposition}
\newtheorem{lem}[thm]{Lemma}

\newtheorem*{thm*}{Theorem}

\theoremstyle{definition}
\newtheorem{defi}[thm]{Definition}

\newtheorem*{NaC}{Notation and Conventions}
\newtheorem{nota}{Notation}

\newtheorem{assu}{Assumption}

\theoremstyle{remark}

\newtheorem*{claim*}{Claim}
\newtheorem*{claimproof*}{Proof of Claim}

\numberwithin{equation}{subsection}

\newcommand{\Z}{\mathbb{Z}}

\newcommand{\R}{\mathbb{R}}
\newcommand{\C}{\mathbb{C}}
\newcommand{\A}{\mathbb{A}}
\renewcommand{\P}{\mathbb{P}}
\newcommand{\F}{\mathbb{F}}
\newcommand{\G}{\mathbb{G}}

\newcommand{\vp}{\varphi}

\newcommand{\mc}{\mathcal}

\newcommand{\tb}{\textbf}

\newcommand{\wt}{\widetilde}

\DeclareMathOperator{\Aut}{Aut}

\DeclareMathOperator{\Sym}{\mathrm{Sym}}

\DeclareMathOperator{\PProj}{\tb{Proj}}
\DeclareMathOperator{\Supp}{Supp}

\DeclareMathOperator{\Pic}{\mathrm{Pic}}

\DeclareMathOperator{\Bl}{\mathrm{Bl}}

\DeclareMathOperator{\Sing}{\mathrm{Sing}}

\title[$\G_{a}^{3}$-structures on del Pezzo fibrations]
{$\G_{a}^{3}$-structures on del Pezzo fibrations}
\author[M. Nagaoka]{Masaru Nagaoka}
\address{Graduate School of Mathematical Sciences\\The University of Tokyo\\3-8-1 Komaba\\Meguro-ku, Tokyo 153-8914, Japan}
\email{nagaoka@ms.u-tokyo.ac.jp}
\begin{document}
\begin{abstract}
In this paper we prove that del Pezzo fibrations admit $\G_{a}^{3}$-structures if and only if they are $\P^2$-bundles over $\P^{1}$.
\end{abstract}
\maketitle
\tableofcontents

\section{Introduction}
\label{sec:intro}


We work over the field of complex numbers $\C$.
Let $\G_{a}^{n}$ be the $n$-dimensional vector group, i.e., the $n$-dimensional affine space $\A^{n}$ equipped with the additive group structure.
In this paper, we are interested in equivariant compactifications of $\G_{a}^{n}$ in the following sense.

\begin{defi}[{\cite[Definition 2.1]{H-T99}}]
Let $\G$ be a connected linear algebraic group.
A $\G$-variety $X$ is a variety with a fixed (left) $\G$-action such that the stabilizer of a general point is trivial and the orbit of a general point is dense. 
\end{defi}

We note that the dense open orbit of a $\G$-variety is isomorphic to $\G$.
By \textit{a $\G$-structure on $X$ with the boundary divisor $D$}, we mean a $\G$-action on $X$ which makes $X$ a $\G$-variety whose dense open orbit is $X \setminus D$.

B.\ Hassett and Y.\ Tschinkel \cite{H-T99} considered $\G_{a}^{n}$-varieties originally, and classified all the smooth projective $\G_{a}^{n}$-varieties with the second Betti number $B_{2}=1$ when $n \leq 3$.
Since smooth rational projective varieties with $B_{2}=1$ are Fano, we can rephrase their result as the classification of all the smooth Fano $\G_{a}^{n}$-varieties with $B_{2}=1$ when $n \leq 3$.
After that, Z.\ Huang and P.\ Montero \cite{H-M} classified all the smooth Fano $\G_{a}^{3}$-varieties with $B_{2} \geq 2$.
B.\ Fu and P.\ Montero \cite{F-M} also classified all the smooth Fano $\G_{a}^{n}$-varieties with Fano index at least $n-2$ for any dimension.

In this paper, we consider smooth projective $\G_{a}^{3}$-varieties with $B_{2}=2$, which are not necessarily Fano.
Take such a variety $X$, which is rational by definition.
By virtue of the Mori theory, it has an extremal contraction $f \colon X \to C$, i.e., a surjective morphism to a normal projective variety $C$ such that $f_{*}\mc{O}_{X} \cong \mc{O}_{C}$, the relative Picard number is one and $-K_{X}$ is $f$-ample.
When $C$ is a curve, we call $f$ \textit{a del Pezzo fibration}.
In this case, \textit{the degree of $f$} is the anti-canonical volume of a general $f$-fiber, which is a del Pezzo surface. 
By \cite[Theorem 3.5]{Mor} the degree is at most nine, and $f$ is a $\P^{2}$-bundle when the degree is nine.

The main theorem of this paper is the following, which classifies smooth projective $\G_{a}^{3}$-varieties with del Pezzo fibration structures.

\begin{thm}\label{thm:maineq}
Let $X$ be a smooth projective $3$-fold, $D$ a reduced effective divisor on $X$ and $f \colon X \to C$ a del Pezzo fibration.
Then the following are equivalent.
\begin{enumerate}
\item[\textup{(1)}] $X$ has a $\G_{a}^{3}$-structure with the boundary divisor $D$.
\item[\textup{(2)}] $f$ is a $\P^{2}$-bundle over $\P^{1}$ and $D$ consists of a sub $\P^{1}$-bundle $D_{1}$ and a $f$-fiber $D_{2}$ which generate the cone of effective Cartier divisors of $X$.
\end{enumerate}
\end{thm}

This paper is structured as follows.
In \S \ref{sec:pre}, we recall some facts on actions of algebraic groups on algebraic varieties and elementary links between $\P^{2}$-bundles.
Using them, we prove that Theorem \ref{thm:maineq} (1) implies (2) in \S \ref{sec:12}.
The main step to prove this implication is Proposition \ref{prop:2}, that is, the exclusion of the case when the degrees of del Pezzo fibrations are eight.
For this, we use the results of \cite{Nag2}.
Finally, we prove the opposite implication in \S \ref{sec:21}.
For that, we construct a $\G_{a}^{3}$-structure for each $\P^{2}$-bundle $P$ over $\P^{1}$ via a sequence of elementary links from $\P^{1} \times \P^{2}$ to $P$.


\begin{NaC}\label{nota:2}
Throughout this paper, we follow \cite[Definition 1.6]{Nag2} for the definition of elementary links.
Also we use the following notation:
\begin{itemize}
\item $\F_{d}$: the Hirzebruch surface of degree $d$.
\item $\Supp Y$: the support of a closed subscheme $Y$ of an ambient variety.
\item $E_{f}$: the exceptional divisor with the reduced structure of a birational morphism $f$.
\item $Y_{\wt X}$: the strict transformation of a closed subscheme $Y$ of a normal variety $X$ in a birational model $\wt X$ of $X$.
\item $\Lambda_{\mathrm{eff}} (X) \subset \Pic (X) \otimes_{\Z} \R$: the cone of effective Cartier divisors on a projective variety $X$.
\end{itemize}
\end{NaC}

\section{Preliminaries}
\label{sec:pre}

In this section, we compile some facts on actions of algebraic groups on algebraic varieties and elementary links between $\P^{2}$-bundles, which will be needed in \S \ref{sec:12} and \S\ref{sec:21}.

\begin{thm}[{\cite[Theorem 2.5, 2.7]{H-T99}}]\label{thm:coeff}
Let $X$ be a normal proper $\G_{a}^{3}$-variety with the boundary divisor $D$ and $D= \cup_{i=1}^{n} D_{i}$ the irreducible decomposition.
Then we have the following:
\begin{enumerate}
\item[\textup{(1)}] $\Pic(X) = \bigoplus_{i=1}^{n} \Z D_{i}$.

\item[\textup{(2)}] $-K_{X} \sim \sum_{i=1}^{n} a_{i} D_{i}$ for some integers $a_{1}, \ldots, a_{n} \geq 2$. 

\item[\textup{(3)}]
$\Lambda_{\mathrm{eff}} (X)=\bigoplus_{i=1}^{n} \R_{\geq 0} D_{i}$.
\end{enumerate}
\end{thm}

\begin{thm}[{\cite[Theorem 7.2.1]{Bri}}]\label{thm:bra}
Let $G$ be a connected algebraic group, $X$ a variety with $G$-action, $Y$ a variety and $f \colon X \to Y$ a proper morphism such that $f^{\sharp} \colon \mc{O}_{Y} \to f_{*}\mc{O}_{X}$ is an isomorphism.
Then there exists the unique $G$-action on $Y$ such that $f$ is equivariant. 
\end{thm}

\begin{thm}[{\cite[Theorem 1.3]{Mar}}]\label{thm:p-p}
Let $p \colon P \to C$ be a $\P^{2}$-bundle and $L \subset P$ a $n$-dimensional linear subspace of a $p$-fiber ($n \leq 1$).
Let $\vp \colon \wt P = \Bl_{L}P \to P$ be the blow-up along $L$.
Then 
\begin{enumerate}
\item[\textup{(1)}] There exists a divisorial contraction $\psi \colon \wt P \to P'$ over $C$ such that the induced morphism $p' \colon P' \to C$ is a $\P^{2}$-bundle and $\psi$ is the blow-up along a $(1-n)$-dimensional linear subspace $L'$ of a $p'$-fiber.
\item[\textup{(2)}] The exceptional divisor $E_{\psi}$ is the strict transform of the $p$-fiber containing $L$. 
\end{enumerate}
\begin{equation}\label{diag:p-p}
\xymatrix@!C=15pt@R=15pt{
&\Bl_{L} P=\wt{P}=\Bl_{L'} P' \ar[dl]_-{\vp} \ar[dr]^-{\psi}
&
\\P \ar[d]_-{p}
&
&P' \ar[d]^-{p'}
\\ C \ar@{=}[rr]
& 
&C.
}
\end{equation}
\end{thm}

\section{Proof of Theorem \ref{thm:maineq} $(1) \Rightarrow (2)$}
\label{sec:12}

In this section, we prove that Theorem \ref{thm:maineq} (1) implies (2).
For this, we make the following assumption in this section:

\begin{assu}
$X$ is a smooth projective $\G_{a}^{3}$-variety with the boundary divisor $D$. 
$f \colon X \to C$ is a del Pezzo fibration of degree $d$.

By Theorem \ref{thm:coeff}, $D$ consists of two irreducible components, say $D_{1} \cup D_{2}$.
\end{assu}

\begin{lem}\label{lem:base}
It holds that $C \cong \P^{1}$.
\end{lem}

\begin{proof}
$X$ is rational since it contains $\G_{a}^{3}$ as the dense open orbit.
Since $H^{0}(C, \Omega_{C}) \hookrightarrow H^{0}(X, \Omega_{X}) \cong 0$, we have $H^{0}(C, \Omega_{C}) \cong 0$ and the assertion holds.  
\end{proof}

\begin{prop}\label{prop:fiber}
The boundary divisor $D$ contains a $f$-fiber which is stable under $\G_{a}^{3}$-action.
\end{prop}

\begin{proof}
By Theorem \ref{thm:bra}, there is the $\G_{a}^{3}$-action on $C$ such that $f$ is $\G_{a}^{3}$-equivariant. 
By the Borel fixed-point theorem \cite[\S 21.2]{Hum}, the action $\G_{a}^{3} \curvearrowright C$ has a fixed point, say $\infty \in C$.
Since the divisor $f^{*}(\infty)$ is stable under the $\G_{a}^{3}$-action, it is contained in $D$.
\end{proof}

Note that each $f$-fiber is irreducible by \cite[Theorem 3.5]{Mor}.
In the remainder of this section we require $D_{2}$ to be a $f$-fiber. 

\begin{prop} \label{prop:1}
It holds that $d \geq 8$.
\end{prop}

\begin{proof}
Conversely, suppose that $d \leq 7$.
By Theorem \ref{thm:coeff} (1), we have $\Pic(X)=\Z D_{1} \oplus \Z D_{2}$.
On the other hand, a general $f$-fiber is a smooth del Pezzo surface of degree $d \leq 7$, which has a $(-1)$-curve, say $l$. 
Combining $(-K_{X} \cdot l)=1$ and \cite[Theorem 3.2]{Mor}, we have $\Pic(X)=\Z(-K_{X}) \oplus \Z D_{2}$.
Hence we can write $-K_{X} \sim a_{1} D_{1}+a_{2} D_{2}$ with $a_{1}=1$ and $a_{2} \in \Z$, a contradiction with Theorem \ref{thm:coeff} (2).
\end{proof}


\begin{prop}\label{prop:2}
It holds that $d \neq 8$.
\end{prop}

\begin{proof}
Conversely, suppose that $d = 8$.

\noindent \underline{Step $1$}:
First we show that we get a contradiction if there is a $\G_{a}^{3}$-stable $f$-section, say $s$.
In this case, we can take the elementary link with center along $s$ by \cite[(2.7.3)]{D'S}:
\begin{equation}\label{diag:elem2}
\xymatrix@!C=15pt@R=15pt{
&\wt{X} \ar[dl]_-{\vp} \ar[dr]^-{\psi}
&
\\X \ar[d]_-{f}
&
&P \ar[d]^-{p}
\\C \ar@{=}[rr]
& 
&C
}
\end{equation}
where $\vp$ is the blow-up along $s$, $p$ is a $\P^{2}$-bundle and $\psi$ is the blow-up along a smooth connected $p$-bisection, say $B$.

Since $s$ is $\G_{a}^{3}$-stable, $\wt X$ admits the unique $\G_{a}^{3}$-action such that $\vp$ is equivariant. 
By Theorem \ref{thm:bra}, $P$ and $C$ also admit the unique $\G_{a}^{3}$-actions such that $\psi$ and $p$ are equivariant respectively.
Since $E_{\psi}$ is $\G_{a}^{3}$-stable, so is $B$.
Hence $p|_{B} \colon B \to C$ is a $\G_{a}^{3}$-equivariant double covering.
Since $X$ has the dense open orbit, so does $C$. 
Since $p|_{B}$ is surjective, finite and $\G_{a}^{3}$-equivariant, $B$ also has the dense open orbit.
Since $C$ and $B$ have dominant maps from $\G_{a}^{3}$, we obtain $C \cong B \cong \P^{1}$.

Let us show that $B$ has the unique $\G_{a}^{3}$-fixed point.
By \cite[Proposition 3.6]{H-M}, $\G_{a}^{3}$ contains a subgroup $G \cong \G_{a}^{2}$ such that the $\G_{a}^{3}$-action on $B$ factorizes via $\G_{a}^{3}/G \cong \G_{a}^{1}$.
Since $\G_{a}^{1}$ has no non-trivial algebraic subgroup, the stabilizer of a general point of this $\G_{a}^{1}$-action is trivial.
Hence this action is a $\G_{a}^{1}$-structure of $B$.
By \cite[Proposition 3.1]{H-T99}, $B$ has the unique fixed point.
By the same argument, $C$ also has the unique $\G_{a}^{3}$-fixed point.

Let $b \in B$ and $c \in C$ are the $\G_{a}^{3}$-fixed points.
Since $p|_{B}$ is equivariant, we have $p(b)=c$.
If $p|_{B}$ is unramified at $b$, then the point in $(p|_{B})^{-1}(c) \setminus \{b\}$ is also fixed, a contradiction.
Hence $p|_{B}$ is ramified at $b$. 
Since $C \cong B \cong \P^{1}$, $p|_{B}$ has the other ramification point, which is also fixed, a contradiction.

\noindent \underline{Step 2}:
Now it suffices to find a $\G_{a}^{3}$-stable $f$-section.
By Theorem \ref{thm:coeff} (2), there are integers $a_{1}, a_{2} \geq 2$ such that $-K_{X} \sim a_{1}D_{1}+a_{2} D_{2}$.
For a smooth $f$-fiber $F \cong \F_{0}$, the restriction $-K_{X}|_{F} \sim a_{1}D_{1}|_{F}$ is a divisor of bidegree $(2,2)$.
Hence $a_{1}=2$.
On the other hand, by the choice of $D_{2}$, $(X, D_{1}, D_{2})$ is a compactification of $\A^{3}$ compatible with $f$ (See \cite[Definition 1.1]{Nag2}).

If $D_{1}$ is non-normal, then $s \coloneqq \Sing D_{1}$ forms a section by \cite[Lemma 2.7]{Nag2}. Since $D_{1}$ is $\G_{a}^{3}$-stable, so is $s$.
Therefore we derive a contradiction as in Step 1.

Hence $D_{1}$ is normal.
By \cite[Theorem 4.2]{Nag2}, $D_{2}$ is isomorphic to the quadric cone.
Recall that in \cite[Definition 7.2]{Nag2}, we assign a non-negative integer to $(X, D_{1}, D_{2})$, which we call the type of the triplet, by using the singularity of $D_{1}$.
By definition, $(X, D_{1}, D_{2})$ is of type $0$ if and only if $D_{1}$ is a Hirzebruch surface.

Suppose that $(X, D_{1}, D_{2})$ is of type $m>0$.
Then $\Supp (D_{1}|_{D_{2}})$ contains a ruling of the quadric cone $D_{2}$ by \cite[Theorem 7.1]{Nag2}, say $l$.
Then we can take the elementary link with center along $l$ by \cite[Lemma 2.6]{Nag2}:
\begin{equation}\label{diag:elem2}
\xymatrix@!C=15pt@R=15pt{
&\wt{X} \ar[dl]_-{\vp} \ar[dr]^-{\psi}
&
\\X \ar[d]_-{f}
&
&X' \ar[d]^-{f'}
\\C \ar@{=}[rr]
& 
&C
}
\end{equation}
where $\vp$ is the blow-up along $l$, $f'$ is a del Pezzo fibration of degree eight and $\psi$ is the blow-up along a ruling in a singular $f'$-fiber such that $E_{\psi}=(D_{2})_{\wt X}$.

Since $\Supp (D_{1}|_{D_{2}})$ is $\G_{a}^{3}$-stable and $\G_{a}^{3}$ is irreducible, $l$ is also $\G_{a}^{3}$-stable.
Hence $\wt X$ admits a $\G_{a}^{3}$-structure with the boundary divisor $(D_{1} \cup D_{2})_{\wt X} \cup E_{\vp}$.
Theorem \ref{thm:bra} now gives $X'$ a $\G_{a}^{3}$-structure with the boundary divisor $(D_{1})_{X'} \cup (E_{\vp})_{X'}$.
By \cite[Theorem 7.5]{Nag2}, $(X', (D_{1})_{X'}, (E_{\vp})_{X'})$ is of type $m-1$.

By repeated application of the above construction, we only have to exclude the case when $(X, D_{1}, D_{2})$ is of type $0$.
Then $D_{1}$ is $\G_{a}^{3}$-stable and is isomorphic to $\F_{n}$ for some $n$.
If $n >0$, then the negative section $s$ in $D_{1}$ is a $\G_{a}^{3}$-stable $f$-section, and we derive a contradiction as in Step 1.
Hence $n=0$.
There is the $\P^{1}$-bundle structure $h \colon D_{1} \to \P^{1}$ other than $f|_{D_{1}}$.
Combining Theorem \ref{thm:bra} and the Borel fixed-point theorem, we get a $\G_{a}^{3}$-stable $h$-fiber $s$, which is a $f$-section.
Therefore we derive a contradiction as in Step 1.
\end{proof}


\begin{proof}[Proof of Theorem \ref{thm:maineq} $(1) \Rightarrow (2)$]
Suppose that (1) holds,
Combining Propositions \ref{prop:1} and \ref{prop:2}, we get $d=9$.
By Theorem \ref{thm:coeff} (2), there are integers $a_{1}, a_{2} \geq 2$ such that $-K_{X} \sim a_{1}D_{1}+a_{2}D_{2}$.
By the adjunction formula, we have $a_{1}D_{1}|_{D_{2}} \sim -K_{X}|_{D_{2}} \sim -K_{D_{2}} \sim \mc{O}_{\P^{2}}(3)$. 
Hence $a_{1}=3$ and $D_{1}$ is a sub $\P^{1}$-bundle.
The second assertion of (2) follows from Theorem \ref{thm:coeff} (3).
\end{proof}

\section{Proof of Theorem \ref{thm:maineq} $(2) \Rightarrow (1)$}
\label{sec:21}

In this section, we prove that Theorem \ref{thm:maineq} (2) implies (1).

\begin{nota}\label{nota:2}
For this, we make the following notation in this section:
\begin{itemize}
\item $\P_{X}(\mc{E}) \coloneqq \PProj_{\mc{O}_{X}} \bigoplus_{m \geq 0} \Sym^{m}(\mc{E})$: the projectivization of a locally free sheaf $\mc{E}$ on a variety $X$. 
\item $\F(e_{1},e_{2},e_{3}) \coloneqq \P_{\P^{1}}(\bigoplus_{i=1}^{3}\mc{O}_{\P^{1}}(e_{i}))$.
\item $p_{d_{1}, d_{2}}$: the $\P^{2}$-bundle structure of $\F(-d_{1}, -d_{2}, 0)$.
\item $\xi_{d_{1}, d_{2}}$: a tautological divisor of $\F(-d_{1}, -d_{2}, 0)$.
\end{itemize}
\end{nota}


To complete the proof of Theorem \ref{thm:maineq}, we prepare the following five lemmas.

\begin{lem}\label{lem:unique}
Let $P \coloneqq \F(-d_{1}, -d_{2}, 0)$ with $d_{1} \geq d_{2} \geq 0$, $E$ a sub $\P^{1}$-bundle of $P$ and $F$ a $p_{d_{1}, d_{2}}$-fiber.
Then $E$ and $F$ generate $\Lambda_{\mathrm{eff}}(P)$ if and only if $E \sim \xi_{d_{1}, d_{2}}$.
Moreover, in this case, the pair $(E, F)$ is unique up to $\Aut(X)$.
\end{lem}

\begin{proof}
Recall from \cite[Chapter 2]{Rei} that $P=\F(-d_{1}, -d_{2}, 0)$ is defined as the quotient of $(\A^{2} \setminus \{0\}) \times (\A^{3} \setminus \{0\})$ by the following $(\G_{m})^{2}$-action:
\begin{eqnarray*}
(\G_{m})^{2} \times (\A^{2} \setminus \{0\}) \times (\A^{3} \setminus \{0\}) &\to& (\A^{2} \setminus \{0\}) \times (\A^{3} \setminus \{0\})\\
((\lambda, \mu), (t_{1}, t_{2} ; x_{1}, x_{2}, x_{3})) &\mapsto& (\lambda t_{0}, \lambda t_{1} ; \lambda^{d_{1}}\mu x_{1}, \lambda^{d_{2}}\mu x_{2}, \mu x_{3}).
\end{eqnarray*}
We also have $\Pic P =\Z \xi_{d_{1}, d_{2}} \oplus \Z F$, and for each $a, b \in \Z$, the linear system $|a\xi_{d_{1}, d_{2}}+bF|$ is parametrized by the vector space of polynomials spanned by monomials $t_{1}^{b_{1}}t_{2}^{b_{2}}x_{1}^{a_{1}}x_{2}^{a_{2}}x_{3}^{a_{3}} \in \C[t_{1}, t_{2}, x_{1}, x_{2}, x_{3}]$ with $a_{1}+a_{2}+a_{3}=a$ and $b_{1}+b_{2}=-d_{1}a_{1}-d_{2}a_{2}+b$.
Hence $|a\xi_{d_{1}, d_{2}}+bF| \neq \emptyset$ if and only if $a \geq 0$ and $b \geq 0$, and the first assertion follows.

Now suppose that $E \sim \xi_{d_{1}, d_{2}}$.
Then $E$ is defined by $\sum_{i=1}^{3}u_{i}x_{i}$ for some $u_{i} \in \C$ for $i=1,2,3$ such that $u_{i}=0$ unless $d_{i}=0$ for $i=1,2$. 
Suppose that $u_{3}=0$. Then $u_{i} \neq 0$ for some $i=1,2$.
Take $\wt h \in \Aut((\A^{2} \setminus \{0\}) \times (\A^{3} \setminus \{0\}))$ which interchanges $x_{i}$ and $x_{3}$, which is $(\G_{m})^{2}$-equivariant.
Since $P$ is the geometric quotient by \cite[Proposition 1.9]{GIT}, it descends to an element in $\Aut(P)$.
Hence we may assume that $u_{3}=1$.
By a similar argument, we also may assume that $F$ is defined by $t_{1}+vt_{2}$ for some $v \in \C$.

Now let $E'$ and $F'$ be divisors on $P$ defined by $x_{3}$ and $t_{1}$ respectively.
Take $\wt h \in \Aut((\A^{2} \setminus \{0\}) \times (\A^{3} \setminus \{0\}))$ such that 
\begin{eqnarray}
&&\wt h^{*}(x_{1})=x_{1}, \wt h^{*}(x_{2})=x_{2}, \wt h^{*}(x_{3})=c_{1}x_{1}+c_{2}x_{2}+x_{3}, \\
&&\wt h^{*}(t_{1})=t_{1}+vt_{2}, \wt h^{*}(t_{2})=t_{2}.
\end{eqnarray}
Since $\wt h$ is $(\G_{m})^{2}$-equivariant, it descends to $h \in \Aut(P)$ such that $h(E)=E'$ and $h(F)=F'$, which complete the proof.
\end{proof}

\begin{lem}\label{lem:lineblowup}
We follow the situation of Theorem \ref{thm:p-p}.
Suppose that $P = \F(-d, -d, 0)$ with $d \geq 0$ and $n=1$.
If there exists $H \in |\xi_{d,d}|$ containing $L$,
then $P' \cong \F(-d-1, -d-1, 0)$ and $H_{P'} \sim \xi_{d+1, d+1}$.
\end{lem}

\begin{proof}
Set $\mc{E}=p'_{*}\mc{O}_{P'}(H_{P'})$. 
It suffices to show that $\mc{E} = \mc{O}_{\P^{1}}(-d-1)^{\oplus 2} \oplus \mc{O}_{\P^{1}}$.
Pushing forward the standard exact sequence
\begin{equation}\label{eq:lineblowup1}
\xymatrix@C=15pt@R=15pt{
0 \ar[r]
&{\mc{O}_{\wt P}(\vp^{*}H-E_{\vp})} \ar[r]
&{\mc{O}_{\wt P}(\vp^{*}H)} \ar[r]
&{\mc{O}_{E_{\vp}}(\vp^{*}H|_{E_{\vp}})} \ar[r]
&0
}
\end{equation}
by $p \circ \vp$, we get the following exact sequence
\begin{equation}\label{eq:lineblowup2}
\xymatrix@C=15pt@R=15pt{
0 \ar[r]
&\mc{E} \ar[r]
&{\mc{O}_{\P^{1}}(-d)^{\oplus 2} \oplus \mc{O}_{\P^{1}}} \ar[r]
&\C^{\oplus 2}\ar[r]
&0
}
\end{equation}
since $\vp^{*}H-E_{\vp} \sim \psi^{*} (H_{P'})$ by Theorem \ref{thm:p-p} (2).
On the other hand, we have $H_{P'} \cong \F_{0}$ because $L \subset H$ and $H \cong \F_{0}$.
By the definition of $\mc{E}$, the inclusion $H_{P'} \subset P'$ corresponds to the exact sequence
\begin{equation}\label{eq:lineblowup3}
\xymatrix@C=15pt@R=15pt{
0 \ar[r]
&\mc{O}_{\P^{1}} \ar[r]
&\mc{E} \ar[r]
&\mc{O}_{\P^{1}}(-a)^{\oplus 2} \ar[r]
&0
}
\end{equation}
for some $a \in \Z$.
Combining $(\ref{eq:lineblowup2})$ and $(\ref{eq:lineblowup3})$, we obtain $-2a=\deg \mc{E}=-2d-2$.
Hence $a=d+1$ and $(\ref{eq:lineblowup3})$ splits, which proves the lemma.
\end{proof}

\begin{lem}\label{lem:lineequiv}
We follow the situation of Lemma \ref{lem:lineblowup}.
Set $\infty \coloneqq p(L) \in C$.
If $P$ admits a $\G_{a}^{3}$-structure with the boundary divisor $H \cup p^{*}(\infty)$,
then so does $P'$ with the boundary divisor $H_{P'} \cup p'^{*}(\infty)$.
\end{lem}

\begin{proof}
Since $L=H \cap p^{*}(\infty)$, this is $\G_{a}^{3}$-stable.
Hence $\wt P$ admits a $\G_{a}^{3}$-structure with the boundary divisor $H_{\wt P} \cup (p \circ \vp)^{*}(\infty)$.
Applying Theorem \ref{thm:bra} to $\psi \colon \wt P \to P'$, we obtain a desired $\G_{a}^{3}$-structure on $P'$.
\end{proof}

\begin{lem}\label{lem:ptblowup}
We follow the situation of Theorem \ref{thm:p-p}.
Suppose that $P = \F(-d_{1}, -d_{2}, 0)$ with $d_{1} \geq d_{2} \geq 0$ and $n=0$.
Assume that there exists $H \in |\xi_{d_{1}, d_{2}}|$ containing $L$, and when $d_{1}>d_{2}$, assume that the negative section of $H \cong \F_{d_{1}-d_{2}}$ passes through $L$ in addition.
Then $P' \cong \F(-d_{1}-1, -d_{2}, 0)$ and $H_{P'} \sim \xi_{d_{1}+1, d_{2}}$.
\end{lem}

\begin{proof}
Set $\mc{E}=p'_{*}\mc{O}_{P'}(H_{P'})$. 
It suffices to show that $\mc{E} = \mc{O}_{\P^{1}}(-d_{1}-1) \oplus \mc{O}_{\P^{1}}(-d_{2}) \oplus \mc{O}_{\P^{1}}$.
By similar arguments as in Lemma \ref{lem:lineblowup}, we get the exact sequence
\begin{equation}\label{eq:ptblowup2}
\xymatrix@C=15pt@R=15pt{
0 \ar[r]
&\mc{E} \ar[r]
&{\mc{O}_{\P^{1}}(-d_{1}) \oplus \mc{O}_{\P^{1}}(-d_{2}) \oplus \mc{O}_{\P^{1}}} \ar[r]
&\C \ar[r]
&0.
}
\end{equation}
Hence $\deg \mc{E}=-d_{1}-d_{2}-1$.
On the other hand, we have $H_{P'} \cong \F_{d_{1}-d_{2}+1}$ by the choice of $L$.
By the definition of $\mc{E}$, the inclusion $H_{P'} \subset P'$ corresponds to the exact sequence
\begin{equation}\label{eq:ptblowup3}
\xymatrix@C=15pt@R=15pt{
0 \ar[r]
&\mc{O}_{\P^{1}} \ar[r]
&\mc{E} \ar[r]
&\mc{O}_{\P^{1}}(-d_{1}-1) \oplus \mc{O}_{\P^{1}}(-d_{2}) \ar[r]
&0.
}
\end{equation}
Since $(\ref{eq:ptblowup3})$ splits, we get the assertion.
\end{proof}

\begin{lem}\label{lem:ptequiv}
We follow the situation of Lemma \ref{lem:ptblowup}.
Set $\infty \coloneqq p(L) \in C$.
If $P$ admits a $\G_{a}^{3}$-structure with the boundary divisor $H \cup p^{*}(\infty)$ such that $L$ is a fixed point,
then so does $P'$ with the boundary divisor $H_{P'} \cup p'^{*}(\infty)$.
\end{lem}

\begin{proof}
Since $L$ is $\G_{a}^{3}$-stable by assumption, we can prove the assertion in much the same way as Lemma \ref{lem:lineequiv}.
\end{proof}

Now we can prove that Theorem \ref{thm:maineq} (2) implies (1).

\begin{proof}[Proof of Theorem \ref{thm:maineq} $(2) \Rightarrow (1)$]

In $\P^{1}_{[t_{1}:t_{2}]} \times \P^{2}_{[x_{1}:x_{2}:x_{3}]}$, set $E \coloneqq \{x_{3}=0\}$ and $F \coloneqq \{t_{1}=0\}$. Write $\infty \coloneqq [0:1] \in \P^{1}$.
Then $E$ and $F$ generate $\Lambda_{\mathrm{eff}}(\P^{1} \times \P^{2})$.
By \cite[Lemma 3.7]{H-M}, $\P^{1} \times \P^{2}$ admits a $\G_{a}^{3}$-structure with the boundary divisor $E \cup F$.
Write this structure as $\rho \colon \G_{a}^{3} \curvearrowright \P^{1} \times \P^{2}$.

Now suppose that $(2)$ follows.
Then $X \cong \F(-d_{1}, -d_{2}, 0)$ for some $d_{1} \geq d_{2} \geq 0$ and $f=p_{d_{1}, d_{2}}$.
By assumption and Lemma \ref{lem:unique}, it holds that $D_{1} \sim \xi_{d_{1}, d_{2}}$ and $D_{2}$ is a $p_{d_{1}, d_{2}}$-fiber.

Suppose that $d_{1}=d_{2}=0$. Then we may assume that $(D_{1}, D_{2})=(E, F)$ by Lemma \ref{lem:unique} and hence $\rho$ is a desired structure.
 
Suppose that $d_{1}=d_{2} > 0$.
Then by Lemma \ref{lem:lineblowup}, we can inductively construct the sequence of the elementary links from $p_{0,0} \colon \P^{1} \times \P^{2} \to \P^{1}$:
\begin{equation}
\xymatrix@C=15pt@R=15pt{
{\P^{1} \times \P^{2}} \ar@{.>}[r]^-{h_{0}} \ar[d]_{p_{0,0}}
&{\F(-1, -1, 0)} \ar@{.>}[r]^-{h_{1}} \ar[d]_{p_{1,1}}
&{\cdots} \ar@{.>}[r]^-{h_{d_{1}-1}} 
&{\F(-d_{1}, -d_{1}, 0)} \ar[d]_{p_{d_{1}, d_{1}}=f} \ar@{=}[r]
&{X}
\\ \P^{1} \ar@{=}[r]
&\P^{1} \ar@{=}[r]
&{\cdots} \ar@{=}[r]
&\P^{1},
&
}
\end{equation}
where the center of $h_{i}$ is the intersection of $E_{i} \coloneqq E_{\F(-i, -i, 0)}$ and $F_{i} \coloneqq p_{i,i}^{*}(\infty)$ for $0 \leq i \leq d_{1}-1$.
Set $E_{d_{1}} \coloneqq E_{X}$ and $F_{d_{1}} \coloneqq f^{*}(\infty)$.
Then $E_{i} \sim \xi_{i,i}$ for $0 \leq i \leq d_{1}$ by Lemma \ref{lem:lineblowup} and hence
we may assume that $(D_{1}, D_{2}) = (E_{d_{1}}, F_{d_{1}})$ by Lemma \ref{lem:unique}.

For $0 \leq i \leq d_{1}-1$, suppose that $\F(-i, -i, 0)$ admits a $\G_{a}^{3}$-structure with the boundary divisor $E_{i} \cup F_{i}$.
Then so does $\F(-(i+1), -(i+1), 0)$ with the boundary divisor $E_{i+1} \cup F_{i+1}$ by Lemma \ref{lem:lineequiv}.
Thus $\rho$ induces a desired $\G_{a}^{3}$-structure on $X$.

Suppose that $d_{1}>d_{2} \geq 0$. 
Set $d=d_{1}-d_{2}$.
Let $\rho'$ be a $\G_{a}^{3}$-structure of $\F(-d_{2}, -d_{2}, 0)$, which we have already constructed. Write its boundary divisor as $E' \cup F'$ such that $E' \sim \xi_{d_{2}, d_{2}}$ and $F'=p_{d_{2}, d_{2}}^{*}(\infty)$. 
By the Borel fixed-point theorem, there is a $\G_{a}^{3}$-fixed point in $E' \cap F'$, say $t_{0}$.
Then by Lemma \ref{lem:ptblowup}, we can inductively construct the sequence of the elementary links from $p_{d_{2}, d_{2}} \colon \F(-d_{2}, -d_{2}, 0) \to \P^{1}$:
\begin{equation}
\xymatrix@C=15pt@R=15pt{
{\F(-d_{2}, -d_{2}, 0)} \ar@{.>}[r]^-{h_{0}} \ar[d]_{p_{d_{2}, d_{2}}}
&{\F(-d_{2}-1, -d_{2}, 0)} \ar@{.>}[r]^-{h_{1}} \ar[d]_{p_{d_{2}+1, d_{2}}}
&{\cdots} \ar@{.>}[r]^-{h_{d-1}} 
&{\F(-d_{1}, -d_{2}, 0)} \ar[d]_{p_{d_{1}, d_{2}}=f}\ar@{=}[r]
&X
\\ \P^{1} \ar@{=}[r]
&\P^{1} \ar@{=}[r]
&{\cdots} \ar@{=}[r]
&\P^{1},
&
}
\end{equation}
where the center of $h_{i}$ is $t_{0}$ for $i=0$ and the intersection of the negative section of $E'_{i} \coloneqq E'_{\F(-d_{2}-i, -d_{2}, 0)} \cong \F_{i}$ and $F'_{i} \coloneqq p_{d_{2}+i, d_{2}}^{*}(\infty)$ for $1 \leq i \leq d-1$.
Set $E'_{d} \coloneqq E'_{X}$ and $F'_{d} \coloneqq f^{*}(\infty)$.
Then $E'_{i} \sim \xi_{d_{2}+i,d_{2}}$ for $0 \leq i \leq d$ by Lemma \ref{lem:ptblowup} and hence
we may assume that $(D_{1}, D_{2}) = (E'_{d}, F'_{d})$ by Lemma \ref{lem:unique}.


Since $t_{0}$ is a fixed point of the action $\rho'$, $\F( -d_{2}-1, -d_{2}, 0)$ admits a $\G_{a}^{3}$-structure with the boundary divisor $E'_{1} \cup F'_{1}$ by Lemma \ref{lem:ptequiv}.

For $1 \leq i \leq d-1$, suppose that $\F(-d_{2}-i, -d_{2}, 0)$ admits a $\G_{a}^{3}$-structure with the boundary divisor $E'_{i} \cup F'_{i}$.
Then $t_{i}$ is a $\G_{a}^{3}$-fixed point by construction.
Hence $\F( -d_{2}-(i+1), -d_{2}, 0)$ admits a $\G_{a}^{3}$-structure with the boundary divisor $E'_{i+1} \cup F'_{i+1}$ by Lemma \ref{lem:ptequiv}.
Thus $\rho'$ induces a desired $\G_{a}^{3}$-structure on $X$.
\end{proof}


\section*{Acknowledgement}

The author is greatly indebted to Professor Hiromichi Takagi, his supervisor, for his encouragement, comments, and suggestions. 
He also wishes to express his gratitude to Professor Adrien Dubouloz for drawing his attention to \cite{H-M}.
He is also grateful to Doctor Katsuhisa Furukawa for his helpful comments.

This work was supported by JSPS KAKENHI Grant Number JP19J14397 and 
the Program for Leading Graduate Schools, MEXT, Japan.

\end{document}